\documentclass[12pt]{amsart}
\allowdisplaybreaks[1]
\usepackage{enumerate}
\usepackage{amsmath}
\usepackage{amssymb}
\usepackage{amsfonts}
\usepackage{verbatim}
\usepackage[dvipsnames]{xcolor}
\usepackage{pgfplots}
\usepackage{mathtools}
\pgfplotsset{compat=1.14}
\usepackage{float}
\usepackage{hyperref}
\makeindex

 \newtheorem{theorem}{Theorem}[section]
 \newtheorem{corollary}[theorem]{Corollary}
 \newtheorem{lemma}[theorem]{Lemma}

\theoremstyle{definition}

\theoremstyle{remark}

\newtheorem{fact*}{Fact}

\DeclareMathOperator{\IM}{Im}

\DeclareMathOperator{\dist}{dist}

\DeclareMathOperator\im{\mathrm {Im~}}

\newcommand\dd{\mathrm d}
\newcommand{\eps}{\varepsilon}

\newcommand{\T}{\mathbb{T}}

\newcommand{\D}{\mathbb{D}}
\newcommand{\C}{\mathbb{C}}

\newcommand{\R}{\mathbb{R}}

\newcommand{\cc}[1]{\overline{#1}}
\newcommand{\abs}[1]{\left\vert#1\right\vert}

\newcommand{\nt}{\stackrel{\mathrm {nt}}{\to}}

\newcommand{\til}{\raise.17ex\hbox{$\scriptstyle\mathtt{\sim}$}}

\newcommand{\ph}{\varphi}

\newcommand\ep{\varepsilon}
\newcommand\la{\lambda}

\newcommand\beq{\begin{equation}}

\newcommand\eeq{\end{equation}}

\newcommand{\bbm}{\left[ \begin{smallmatrix}}
\newcommand{\ebm}{\end{smallmatrix} \right]}
\newcommand{\bpm}{\left( \begin{smallmatrix}}
\newcommand{\epm}{\end{smallmatrix} \right)}
\numberwithin{equation}{section}

\newlength{\Mheight}
\newlength{\cwidth}

\newcommand{\dfn}[1]{{\bf #1}\index{#1}}

\newcommand{\lapl}{\bigtriangleup}

\title[Escaping nontangentiality]{Escaping nontangentiality: Towards a controlled tangential amortized Julia-Carath\'eodory theory}
\author{
J. E. Pascoe$^\dagger$
}
\address{Department of Mathematics\\
1400 Stadium Rd\\
  University of Florida\\
 Gainesville, FL 32611}
\email[J. E. Pascoe]{pascoej@ufl.edu}

\thanks{$\dagger$ Partially supported by National Science Foundation Mathematical
Science Postdoctoral Research Fellowship  
DMS 1606260}

\author{
Meredith Sargent
}
\address{
Department of Mathematical Sciences\\
309 SCEN\\
University of Arkansas\\
Fayetteville, AR 72701
}
\email[M. Sargent]{sargent@uark.edu}
\author{
Ryan Tully-Doyle$^\ddagger$
}
\address{Department of Mathematics and Physics \\
University of New Haven\\
West Haven, CT 06516 }
\email[R. Tully-Doyle]{rtullydoyle@newhaven.edu}
\thanks{$\ddagger$ Partially supported by University of New Haven SRG Grant}

\date{\today}


\begin{document}

%Boundary behavior of bounded analytic functions, Julia-Caratheodory theorem, Pick Functions, Moment Problems, gamma regularity
\begin{abstract}
	Let $f: D \rightarrow \Omega$ be a complex analytic function.
	The Julia quotient is given by the ratio between the distance of $f(z)$ to the boundary of $\Omega$ and the distance 
	of $z$ to the boundary of $D.$
	A classical Julia-Carath\'eodory type theorem states that if there is a sequence tending to $\tau$
	in the boundary of $D$ along which the Julia quotient is bounded, then the function $f$
	can be extended to $\tau$ such that $f$ is nontangentially continuous and differentiable
	at $\tau$ and $f(\tau)$ is in the boundary of $\Omega.$
	We develop an extended theory when $D$ and $\Omega$ are taken to be the upper half plane which corresponds to
	amortized boundedness of the Julia quotient on sets of controlled tangential approach,
	so-called $\lambda$-Stolz regions, and higher order regularity, including but not limited to
	higher order differentiability, which we measure using $\gamma$-regularity.
	Applications are given, including perturbation theory and moment problems.
\end{abstract}

\maketitle
\tableofcontents

\section{Introduction}

	Let $D$ and $\Omega$ be open proper subsets in $\C$.
	Let $f: D \to \Omega$ be a complex analytic function.
	The \dfn{Julia quotient} of a function 
	evaluated at $z \in D$ is the ratio of distances
	\[
		J_f(z) = \frac{\dist(f(z), \partial \Omega)}{\dist(z,\partial D)}. 
	\]

	When $D=\Omega= \Pi,$ where $\Pi$ denotes the upper half plane in $\C,$ the Julia quotient is given by the formula
		\beq \label{JCUHP} \frac{\im f(z)}{\im z},\eeq
	where $\im w$ denotes the imaginary part of a complex number $w.$
	Note that Equation \eqref{JCUHP} is linear in $f$ over $\mathbb{R}$.
	
	We define a \dfn{Stolz region} in a domain $D$ at a point $\tau \in \partial D$ with aperture $M,$ denoted $S_{\tau,M}$, to be the set
		$$S_{\tau ,M} = \{z\in D|\dist (z, \partial D) \geq M \dist(z, \tau)\}.$$
	Note that for the Stolz region $S_{\tau,M}$ to be nonempty, we must have $M\leq 1.$

	In many classically important domains, complex analytic functions $f$ with the property that the Julia quotient is bounded along some sequence approaching
	to $\tau$ possess certain strong regularity -- specifically  continuity and differentiability on the closure of each 
	Stolz region, $\cc S_{\tau,M}$. Moreover, the boundary value $f(\tau)$ must be in $\partial \Omega.$ 
	G. Julia \cite{ju20} and C. Carath\'eodory \cite{car29} classically developed the theory on the unit disk. On the upper half plane, the theory was developed
	by R. Nevanlinna \cite{nev22}. Numerous modern treatments exist on domains in several variables and more rigid forms of regularity, e.g.
	\cite{ab98, jaf93, lugarnedic, amy10a, pascoePEMS, mprevisit}.

	Let $\lambda:[0,\infty)\rightarrow \R^{\geq 0}$ be a function. 
	We define a \dfn{$\lambda$-Stolz region at $\tau$} to be the set 
	\begin{align*} 
		S^\lambda_{\tau}= \big\{z\in D |  &\dist (z, \partial D) \geq \lambda(C), \\
		&\sqrt{\dist(z, \tau)^2 -\dist (z, \partial D)^2} \leq C, \text{for some } C>0\big\}.
	\end{align*}
	Note that a classical Stolz region with parameter $0 < M \leq 1$
	is a $\lambda$-Stolz region where $\lambda(t) = \left(\sqrt{\frac{M^2}{1-M^2}}\right)t.$
	However, we will be particularly interested in the case where $\lambda(t)$ is $o(t).$
	
	A rudimentary tangential Fatou type theory has been developed for certain $\la$-Stolz regions by
	Nagel, Rudin, and Shapiro \cite{nrs1982} and Nagel and Stein \cite{ns84},
	in spite of a failure for the presumably tame class of bounded analytic functions as noted
	by Littlewood \cite{littlewood1927} and Zygmund \cite{zygmund49} classically. 
	For a survey, see \cite{ks17}.

\begin{figure}
	\caption{A classical Stolz region is depicted in blue overlayed on a $e^{-1/t}$-Stolz region depicted in red. Since the 
	$e^{-1/t}$-Stolz region contacts the boundary much more tightly than the classical Stolz region, (amortized) boundedness of the Julia quotient there implies significantly more regularity.}
	\begin{tikzpicture}[scale=.7]
    \begin{axis}[
        domain=-1:1, %here is where you mess with the domain. Change here and in the next line
        xmin=-.5, xmax=.5,
        ymin=-0, ymax=.13, %vhange the range here
        samples=401,
        axis y line=none,
        axis x line=middle,
	ticks=none
    	]
        %\addlegendentry{Stolz region}
        \addplot+[mark=none, black, fill=Red] {e^(-(1/abs(x)))};
        %\addlegendentry{$e^{-1/|t|}$-Stolz}
	\addplot+[mark=none, black, fill=Blue] {abs(x)/e};
    \end{axis}
\end{tikzpicture}
\end{figure}
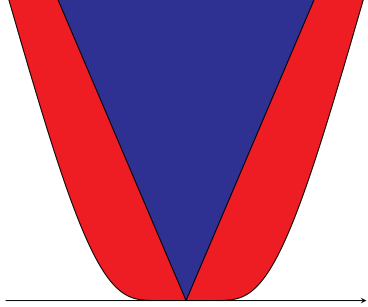

Let $f: \Pi \rightarrow \cc \Pi$ be an analytic function.
Define the \dfn{Nevanlinna measure} $\mu_f$ corresponding to $f$ as the weak limit
of the measures (in $x$) given by $\im f(x+iy)$ as $y \rightarrow 0.$ Where it is unambiguous,
we will often use drop the $f$ and write $\mu.$ The measure $\mu_f,$ with some additional scalar data, can be used to recover $f$ as in Theorem 
\ref{nevrep}.

Let $\gamma: [0,\infty)\rightarrow \mathbb{R}^{\geq 0}$ be a monotone increasing function such that $\gamma(t)$ is $O(t^2)$ as $t \to 0$.
We say an analytic function $f: \Pi \rightarrow \cc \Pi$
is \dfn{$\gamma$-regular at $\tau$} whenever there exists a $C>0$ such that $\frac{1}{\gamma(C|t-\tau|)}$ is integrable on a neighborhood of $\tau$
with respect to the Nevanlinna measure $\mu_f.$ Integrability of certain functions against $\mu_f$ is classically important, e.g. \cite{nev22}.
For any function which is $\gamma$-regular at $\tau$, there exists a $C>0$ such that $f$ must be bounded on sets of the form
$S_\tau^{D\gamma(Ct)}\cap B(\tau,1/D)$ for all $D > 0.$ Moreover, as $D \rightarrow \infty,$ the value on these sets must
go to the nontangential value $f(\tau),$ as is demonstrated in Theorem \ref{horotheorem}. We give an analysis and interpretation of $\gamma$-regularity in Section \ref{pitting}.

To understand the relationship between boundedness of the Julia quotient on $\la$-Stolz regions and $\gamma$-regularity,
one must tame the chaotic behavior of $f$ on the tangential part of the $\la$-Stolz region by averaging or \emph{amortizing} the Julia quotient. Given a $\la$-Stolz region $S_\tau^\la$ at $\tau$, we consider the behavior of the average value of $J_f$ along arcs of constant distance $\la(d)$ from $\partial D$ in $S_\tau^\la$ as $d \to 0$. 

\begin{figure}
\caption{An upper half plane $\la$-Stolz region with a representative curve $C_d$.}
\begin{tikzpicture}[scale = .7]
    \begin{axis}[
        domain=-1:1,
        xmin=-1, xmax=1,
        ymin=0, ymax=.3,
        samples=201,
        axis y line=center,
        axis x line=bottom,
        legend style={at={(1,.7)},anchor=north west}, ticks=none
    	]
        \addplot+[mark=none, thick, black, fill=Goldenrod, opacity = .2] {e^(-1/abs(x))};
        \node[label={0:{$(d,\lambda(d))$}},circle,fill,inner sep=2pt] at (axis cs:.5,e^-2) {};
        \draw [dashed, thick] (-.5,e^-2) -- (.5,e^-2);
        \node[label={90:{$C_d$}}] at (axis cs:-.25, e^-2) {};
    \end{axis}
\end{tikzpicture}
\end{figure}
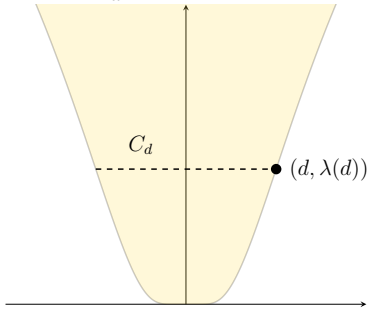

For a given distance $d > 0$, let 
$$C_d = \{z : \dist(z, \partial D) = \la(d), \sqrt{\dist(z, \tau)^2 -\dist (z, \partial D)^2} \leq d\}
.$$ (Essentially, the arcs $C_d$ decompose $S_\tau^\la$.) The \dfn{amortized Julia quotient} of $f$ with respect to $\lambda$ at $\tau$, denoted  $AJ_{f,\lambda}^\tau(d)$, is defined as
\begin{equation}\label{defAJ}
	AJ_{f,\lambda}^\tau(d) = \frac{1}{\abs{C_d}} \int_{C_d}  J_{f}(z) \, \dd|z|,	
\end{equation}
where $\abs{C_d}$ denotes arclength.

To determine $\gamma$-regularity, one needs to determine the existence
of a $\lambda$ such that the amortized Julia quotient $AJ_{f,\la}^\tau(d)$ is
well-behaved and also has the property of being a so-called \emph{$\gamma$-augury}.

We say a function $\lambda: [0,\infty) \rightarrow \mathbb{R}^{\geq 0}$ such that $\lambda(t)$ is $o(t)$ as $t\rightarrow 0$ is a \dfn{$\gamma$-augury} if there exists a $C>0$ such that
    $$\frac{t\lambda(Ct)\dd\gamma(t)}{\gamma(t)^2}$$
is integrable on $[0,1).$ (Here $\dd\gamma(t)$ denotes the distributional derivative of $\gamma.$)

In particular, a $t^n$-augury must have
    $$\frac{t\lambda(t)nt^{n-1}}{t^{2n}} = \frac{\lambda(t)}{t^n}$$
integrable. So, for example, $\lambda(t) = t^{n-1+\varepsilon}$ is a $t^n$-augury.

\begin{theorem}\label{mainthm}
Let $\gamma$ be  $O(t^2).$
\begin{enumerate}
\item An analytic function
$f: \Pi \rightarrow \cc \Pi$
is $\gamma$-regular
if and only if there exists 
a $\gamma$-augury $\lambda$ such that
$AJ_{f,\la}^\tau(d)$ is bounded as $d \rightarrow 0.$
\item In addition to (1), suppose that $\sqrt{\gamma(t)}$ is a 
        $\gamma$-augury.
	Let $f: \Pi \rightarrow \cc \Pi$ be an analytic function. The following are equivalent:
	\begin{enumerate}
		\item the function $f$ is $\gamma$-regular,
		\item $J_{f}(z)$ is bounded on 
                $S^{\sqrt{D\gamma(Ct)}}_{\tau}$ as $z\rightarrow 0$ for some $C,D>0,$
		\item
		$$\limsup_{t\rightarrow 0} J_{f}(\tau+ t+ iD\gamma(C|t|)) < \infty.$$
	\end{enumerate}
\item In addition to (1), suppose that $\gamma(t)/t$ is a 
        $\gamma$-augury.
            An analytic function
                $f: \Pi \rightarrow \cc \Pi$
                is $\gamma$-regular
                if and only if
                $AJ_{f,\gamma(Ct)/Ct}^\tau(d)$ is bounded
                %on 
                %$S^{\gamma(Ct)/Ct}_{\tau}$
                as $d\rightarrow 0$ for some $C>0.$
\end{enumerate}
\end{theorem}
Theorem \ref{mainthm} corresponds $\gamma$-regularity with boundedness on a certain
$\la$-Stolz region. In general, as is in the case of Part (1), the choice of $\la$ evidently depends intensely on the function being analyzed. 
    Part (2) immediately gives a strengthening of our main result for functions with rapid decay, for example $\gamma(t)=e^{-1/t}.$ For such $\gamma,$ the amortization procedure is unnecessary. We view this
as a non-amortized or ``perfect'' Julia-Carath\'eodory type theorem.
The ``perfect" nature of Part (2) allows it to be carried to the disk,
as in Corollary \ref{supercordisk}.
 In
Part (3), which represents a weaker regime than Part (2), we obtain the ability to preordain the choice of $\gamma$-augury, as opposed being forced to artisanally construct a $\gamma$-augury $\lambda$.

Theorem \ref{mainthm} is proven in several parts. Part (1) is proven as Theorem \ref{mainthmPrime}. Part (2) is given in Corollary \ref{supercor}. Part (3) is given in Corollary \ref{uniformcor}.

\section{Motivation and application}

	{
	When $D=\Omega= \D,$ the unit disk in $\C,$ the Julia quotient is given by the formula
		$$\frac{1-|\ph(z)|}{1-|z|}.$$
	}
	{
	A set $S \subset D$ is \dfn{nontangential} at $\tau \in \partial D$ if  $S \subseteq S_{\tau, M}$ for some value of $M.$
	A statement is said to hold \dfn{nontangentially} at $\tau$ if it is true for all $S_{\tau, M}.$}

{
The original theorems of Julia and Carath\'eodory extend Fatou's result on the existence of nontangential boundary limits to describe when an analytic function $\ph: \D \to \cc \D$ has a conformal linear approximation at a boundary point $\tau \in \D$. 
\begin{theorem}[Julia-Carath\'eodory]\label{jc}
Let $\ph: \D \to \cc\D$ be an analytic function. Let $\tau$ be a point in $\T = \partial \D$. The following are equivalent:
\begin{enumerate}
    \item There exists a sequence $\la_n \subset \D$ tending to $\tau$ such that $J_\ph(\la_n)$ is bounded as $\la_n \to \tau$;
    \item for every sequence $\la_n$ in $\D$ tending to $\tau$ nontangentially, the sequence $J_\ph(\la_n)$ is bounded;
    \item the function $\ph$ has a conformal linear approximation near $\tau$. That is, there exist an $\omega = \ph(\tau) \in \T$ and an $\eta = \ph'(\tau) \in \C$ so that 
    \[
    \ph(z) = \omega + \eta(z - \tau) + o(\abs{z - \tau}).
    \]
    as $\la \nt \tau$.
\end{enumerate}
\end{theorem}
}

{
The set of analytic maps from the upper half plane into itself is called the \dfn{Pick class} (variously also known in the literature as the Carath\'eodory class or the Nevanlinna class). Functions $f$ in the Pick class are conformally related to functions on the disk (the Schur class) by the Cayley transform  $\D \to \Pi$ given by
\[
 z\in \Pi = i\frac{1 + \la}{1 - \la}, \hspace{1cm} \la  \in \D.
 \]
This frequently allows results on the disk to be brought to bear on the Pick class. 
}

{
Our main tool on the upper half plane is the classical \dfn{Nevanlinna representation}.
\begin{theorem}[Nevanlinna Representation]\label{nevrep}
	Let $f: \Pi \to \C.$
	The function $f$ is analytic and maps $\Pi$ to $\cc \Pi$ if and only if there exist $a \in \R$, $b \geq 0$ and a positive Borel measure $\mu$ on $\R$ where $\frac{1}{1+t^2}$ is integrable such that
	\beq\label{fullnev}
		f(z) = a + bz + \int_\R \frac{1}{t-z} - \frac{t}{1+t^2} \dd\mu(t)
	\eeq
	for all $z \in \Pi$.
\end{theorem}
The representing measure in Theorem \ref{nevrep} is exactly the Nevanlinna measure $\mu_f$ up to a factor of $\pi$. The expression $\int \frac{1}{t -z} \, d\mu(t)$ in the integral above is the well-known \emph{Cauchy transform} of $\mu$.

{
We call the existence of a conformal linear approximation near $\tau$ ``regularity to order $1$ at $\tau$."
We say a function $f$ is \dfn{regular to order $n$} at $\tau \in \partial D$ if $$f(z) = \tilde{f}(z) + o(d(z,\tau)^{n}) \text{ nontangentially}$$
where $\tilde{f}$ is a map defined on a neighborhood of $\tau$ which takes 
$\partial D$ to the boundary $\partial \Omega.$ (On the disk, $\tilde{f}$ can be taken to be a Blaschke product. On the upper half plane,
$\tilde{f}$ can be taken to be a polynomial with real coefficients.). The following corollary characterizes higher order regularity in terms of the measure $\mu$ arising in the Nevanlinna representation.}
\begin{corollary}\label{nevcor}
An analytic function
$f: \Pi \rightarrow \cc \Pi$
with Nevanlinna representation 
$$f(z) = a + bz + \int_\R \frac{1}{t-z} -\frac{t}{1+t^2} \dd\mu(t)$$
is regular to order $2n-1$ at $\tau$ if and only if 
$\frac{1}{(t-\tau)^{2n}}$ is integrable with respect to $\mu.$
\end{corollary}
Corollary \ref{nevcor} follows from Theorem \ref{nevrep} as a basic exercise in measure theory and algebraic manipulation of integrals. The key step is the following calculation, which writes $f$ as a polynomial plus an object that vanishes nontangentially to order $2n-1.$
%{\blue fix this}
\begin{align*}
        f(z) &= a + bz + \int_\R \frac{1}{t-z} - \frac{t}{1+t^2} \dd\mu(t) \\
		f(z) & = a + bz + \int_\R \left[\frac{1}{t-z} - \frac{t}{1+t^2}\right]t^{2n} \frac{\dd\mu(t)}{t^{2n}} \\
		f(z) & = a + bz + \int_\R \frac{t^{2n}-z^{2n}}{t-z} - \frac{t^{2n+1}}{1+t^2} \frac{\dd\mu(t)}{t^{2n}}
		+z^{2n}\int_\R \frac{1}{t-z} \frac{\dd\mu(t)}{t^{2n}}.
\end{align*}
Note that $ \int_\R \frac{t^{2n}-z^{2n}}{t-z} - \frac{t^{2n+1}}{1+t^2} \frac{\dd\mu(t)}{t^{2n}}$ is a polynomial of order $2n -1$ in disguise.
{ Working on rational functions in two variables on $\mathbb{D}^2$, Bickel, Pascoe and Sola \cite[Theorem 7.1]{bps1} applied the method of Hankel vector moment sequences as developed in \cite{amchvms, pascoePEMS} to give a concrete relationship between regularity to order $n$ and confinement of singular behavior to regions which approach the boundary with comparable rate. We demonstrate a concrete relationship between boundedness of the Julia quotient on sets which approach the boundary with some rate and regularity of the function nontangentially, which served as foundational inspiration for the current enterprise.}

Regularity to order $n$ was analyzed in terms of the iterated Laplacian of the Julia quotient by Bolotnikov and Kheifets in \cite{bk06}.
We give an extended theory along the lines of Theorem \ref{mainthm} in Section \ref{iterated}.

\subsection{Extended Julia-Carath\'eodory theorems on the disk}

For simply connected domains $D, \Omega \subsetneq \mathbb{C},$ we say $f:D\rightarrow \Omega$ is \dfn{transform $\gamma$-regular}
whenever there are conformal maps $\psi,\chi$ such that $\psi \circ f \circ \chi$ is $\gamma$-regular.
The classical Julia inequality \cite{ju20} for maps $\ph$ from $\mathbb{D}$ to $\cc{\mathbb{D}}$ states:
\beq \label{thejuliainequality}
\frac{|\varphi(z) - \omega |^2}
{1-|\varphi(z)|^2}
\leq
\alpha \frac {|z - \tau|^2} 
{1-|z|^2},\eeq
where $\alpha$ is given by the formula
$\liminf_{z\to\tau} \frac{1-|\varphi(z)|^2}{1-|z|^2} = \alpha.$
One obtains that 
	$$ 
\frac{|\varphi(z) - \omega |^2}{|z - \tau|^2}
\leq
\alpha \frac {1 - |\varphi(z)|^2} {1-|z|^2} \leq 2 \alpha J_{f}(z),$$
and so the function $\ph$ restricted to a subset $S$ is continuous at $\tau$ whenever the Julia quotient for $\ph$ is bounded on $S.$
Thus, one may suitably M\"obius transform between the disk and upper half plane freely and obtain the following form of Theorem \ref{mainthm} Part 2.
\begin{corollary}\label{supercordisk}
 Let $\gamma$ be  $O(t^2)$ such that $\sqrt{\gamma(t)}$ is a 
        $\gamma$-augury.
	Let $\varphi: \mathbb{D} \rightarrow \cc{\mathbb{D}}$ be an analytic function. The following are equivalent:
	\begin{enumerate}
		\item the function $\ph$ is transform $\gamma$-regular,
		\item $J_{\ph}(z)$ is bounded on 
                $S^{\sqrt{D\gamma(Ct)}}_{\tau}$ as $z\rightarrow 0$ for some $C,D>0,$
		\item
		$$\limsup_{t\rightarrow 0} J_{\ph}(\tau e^{-D\gamma(C|t|)+it}) < \infty.$$
	\end{enumerate}
    \end{corollary}

\subsection{Perturbation theory}

{
The connection between the function and properties of representing measure $\mu_f$ can be turned into a detailed analysis of the spectrum of a self-adjoint operator via the spectral theorem using a combination of classical and modern techniques. See N. Nikolski's \cite{nikolski1}, particularly Chapter 3, for an extensive survey of the subject background. 

In \cite{lt2016, lt2017}, for example, C. Liaw and S. Treil use Cauchy-type transforms (essentially Nevanlinna representations with restricted measures) to explore rank one perturbations of self-adjoint and unitary operators, recovering information about the operators from the functions that arise from their representing measures. 

Given a self-adjoint operator $A$ and a positive rank one operator $P,$
one may want to understand the spectrum and moreover the spectral measures of $A + \alpha P.$ Specifically, the rank one perturbations give rise to a family of analytic functions $F_\alpha: \Pi \rightarrow \cc \Pi$ such that 
$F_\alpha = \frac{F}{1+\alpha F}$ (the so-called \emph{Aronszajn-Krein formula}), and the structure of the representing measure for $F_\alpha$
can be understood to be somewhat robust under perturbation. For example, the so-called Kato-Rosenblum and Aronszajn-Donoghue theorems say that
the absolutely continuous and singular spectrum are essentially preserved.
Our Theorem \ref{perturbationtheorem} shows that for most reasonable $\gamma$,
$\gamma$-regularity is conformally invariant, and therefore, for example, factors through the Aronszajn-Krein formula.

  The notion of $\gamma$-regularity as developed here gives a detailed analysis of the ``edge of the support" of a measure $\mu.$ Much of the fundamental algebraic structure, which is essentially established in Lemma \ref{AJmuBound}, can (in principle) be adapted to understand other parts (absolutely continuous, singular spectrum etc.) with a modified Julia type quotient of the form $\frac{\dist(\ph(z), \partial \Omega)}{\dist(z, \partial D)^k}$ for $-1 \leq k \leq 1.$
}

\subsection{Applications to moment problems and a problem} \label{momentSection}
We now give several examples and applications of our results to moment problems. First and most obviously, we see that
$t^n$-regularity is equivalent to the existence of absolute moments
$\int |t|^{-n}\dd\mu.$ { More broadly, the notion of $\gamma$-regularity is connected to the classical Hamburger moment problem and the related finite moment problem. See, e.g. \cite{akhiezer}, Theorem 2.1.1 and Theorem 3.2.1 for a discussion of the classical problem and the developments related to its solution, as well as \cite{amchvms, pascoePEMS} for a two-variable generalization. The general approach of analyzing nontangential regularity by  examining a reduction of the representation in Equation \eqref{fullnev} to a Cauchy transform is reminiscent of the approach we take in the present work.
}

Perhaps more interesting is the connection to the so-called \dfn{moment determinacy problem}. Specifically,
when do the moments $m_n = \int t^{n}\dd\mu$ determine the measure $\mu$ uniquely? In our case, we will be interested in the negative moment problem determinacy, i.e. when does the sequence  $\int t^{-n}\dd\mu$ determine $\mu$ uniquely? Certain moment determinacy conditions translate directly into certain forms of $\gamma$-regularity.

We say a measure $\mu$ is \dfn{analytically determined} if the sequence of moments satisfies $|m_{n}|< AB^nn!.$
The condition of analytic determinacy implies that the moments of $\mu$ uniquely determine $\mu$. (See, for example, the discussion of the classical Hamburger moment problem in Sections $X.1$ and $X.6$ of \cite{reedsimon}, particularly Example 4 of Section $X.6$.)
Additionally, we see that the condition is equivalent to the condition that the series 
$\sum \frac{\int t^n \dd \mu(t)}{n!}x^n$ has positive radius of convergence,
which is precisely equivalent to saying that $e^{C/t}$ is $\mu$-integrable for some $C>0.$
Tracing through the definitions, we see the following.
\begin{theorem}
    A function $f$ is $e^{-1/t}$-regular at $0$ if and only if
    the inverse moment problem for $\mu_f$ is analytically determined.
\end{theorem}

We say a measure $\mu$ is \dfn{quasi-analytically determined} if the sequence of moments satisfies $\sum m_{2n}^{-2n}$ diverges.
The condition of quasi-analytic determinacy implies that the moments of $\mu$ uniquely determine $\mu$. (This is the so-called Carleman condition. Again, see  \cite{reedsimon}.)
In this case, test functions of the form 
$g_p(x) =\sum \frac{\int t^{2n} \dd \mu(t)}{p_{2n}}x^{2n}$,
where $p_n$ satisfies the condition, serve the role of $e^{-1/x}.$
That is: if a function $f$ is $e^{g_p(1/t)}$-regular at $0,$ then
    the inverse moment problem for $\mu$ is quasi-analytically determined.
Whenever  $p_n = (nd(n))^n$ where $d$ is monotone 
and  $e^{Ax/d(Bx)} \leq g_{nd(n)}(x) \leq  e^{A'x/d(B'x)}$, we say $p_n$ is \dfn{pseudo-Denjoy}.
One can show, in a somewhat involved exercise, for example, that $p_n = n\ln n$ is pseudo-Denjoy.
Moreover, whenever $p_n$ is pseudo-Denjoy, we have that $\sum m_{2n}^{-2n}$
diverges if and only if $x[(d(1/x)/x)']$ is not integrable on $[0,1]$, via a calculation using Riemann sums.

One can ask, for a given $\gamma$ and $\gamma$-augury $\la$ with the additional property that $\frac{t \la(t)}{\gamma(t)^2} \dd\gamma(t)$ is integrable (that is, $C=1$ in the definition of $\gamma$-augury), \emph{when can the quotient $\la/\gamma$ be bounded?} What follows is a somewhat informal heuristic.
Let $\gamma(t) = e^{\kappa(t)}.$
Let $\lambda(t) = \eta(t) e^{\kappa(t)}$
So,
    $$\frac{t \la(t)}{\gamma(t)^2} \dd\gamma(t) = t \eta(t) \dd \kappa(t)$$
Formally, we could let $\eta(t) = \frac{\zeta(t)}{t \dd \kappa(t)}.$
So we need that $\zeta(t)$ is integrable.
Now, $\la(t)/\gamma(t) = \frac{\zeta(t)}{t \dd \kappa(t)}.$
So we see that if the quotient is to be bounded, then $t \,\dd \kappa(t)$
should be integrable.

The parallel with the pseudo-Denjoy case would give a na\"ive conjecture that
the moment problem is determined whenever $f$ is $e^{\kappa(t)}$-regular
for some $\kappa$ such that $t \, \dd \kappa(t)$ is non-integrable on $[0,1].$ (Moreover, that perhaps some partial or whole converses are true.)

\section{Elementary properties of the Julia quotient for Pick functions}

We now establish some elementary facts about maps from the upper half plane into itself. 
The Pick class has the useful property (notably absent for self-maps of the disk) that it forms a cone.
In the following discussion, we will show that the cone structure carries
through to the associated Julia quotients and derived objects that we use in our proof of the main result.

Thus, we often work at the boundary point $\tau = 0$ without any loss of generality. On the upper half plane, $\Pi,$
the Julia quotient is of the form
\beq\label{pickjulia} J_f(z) = \frac{\im f(z)}{\im z}.\eeq

The lemma follows immediately from Equation \ref{pickjulia}.
\begin{lemma}\label{jadd} \label{ajadd}
Let $f, g: \Pi \to \cc\Pi$ be analytic functions. Then,
\[
J_{f+g}(z) = J_f(z) + J_g(z).
\]
Consequently,
\[
AJ^\tau_{f+g} = AJ^\tau_f + AJ^\tau_g.
\]
\end{lemma} 

 On the upper half plane the Nevanlinna representation is also additive (prima facie, there is no reason to believe that the sum of representations should be the representation of the sum). 
\begin{lemma}
    Let $f,g: \Pi \to \cc\Pi$ be analytic functions with Nevanlinna representations against $\mu_f$ and $\mu_g$ respectively. Then 
    $f+g$ has a Nevanlinna representation against the measure $\mu_f + \mu_g$.
\end{lemma}
\begin{proof}
This follows directly from the definition of Nevanlinna measure, as the weak limit of $\im f(x+iy) + \im g(x+iy)$ is the sum of the limits.
\end{proof}

The following Lemma, though apparently somewhat vacuous, codifies an important fact about $AJ^\tau_{f, \la}(d)$ away from $0$.
\begin{lemma}\label{vacuous}
    Let $f: \Pi \to \cc\Pi$ be an analytic function with Nevanlinna representation measure $\mu$ such that $(-1,1)$
    is not in the support of $\mu.$
    Then, $J_f$ is continuous on $\D$.
    Consequently,
    $AJ^0_{f,\lambda}(d)$ is bounded as $d \rightarrow 0.$
\end{lemma}
\begin{proof}
A quick analysis of Equation \eqref{fullnev} gives that $f$ is analytic on $\mathbb{D}.$ The function $\im f$ is a real analytic function of two variables on $\D$,
and $\im f |_{\mathbb{R}\cap \D}$ is equal to $0,$ and in fact the power series at $0$ converges on all of $\D.$ Therefore, $\im f = (\im z)u$
for some real analytic $u.$ Now note that $J_f = u.$
\end{proof}

The Nevanlinna representation of a Pick function $f$ allows a decomposition of $f$ into the sum of Pick functions, concentrated near the point of interest, in this case $\tau = 0$. The core of this decomposition is a restricted Pick function $f_{red}$ that encodes the behavior of $f$ near $0$ in an algebraically simpler form (in fact a Cauchy transform). The argument about the decomposition of $f$ that follows is in the spirit of the proof given for Theorem 5 in Ch. 32 of \cite{lax02}.

\begin{lemma}\label{splitupterms}
    Let $f: \Pi \to \cc\Pi$ be an analytic function.
    There are analytic functions 
    $f_{trivial}, f_{red}: \Pi \to \cc\Pi$ such that:
    \begin{enumerate}
        \item $f = f_{trivial} + f_{red},$
        \item $f_{trivial}$ is analytic on $\mathbb{D}$ and
        real valued on $(-1,1),$
        \item
            $f_{red} = \int \frac{1}{t-z} d\tilde{\mu}(t)$
        where $\tilde{\mu} = \mu|_{[-1,1]}.$ 
    \end{enumerate}
    Moreover,
    \begin{itemize}
        \item[(A)] $AJ^0_{f, \lambda}(d)$ is bounded as $d \rightarrow 0$ if and only if 
        $AJ^0_{f_{red}, \lambda}(d)$ is bounded as $d\rightarrow 0.$
        \item[(B)] If $\gamma: [0, \infty) \to \R^{\geq 0}$ is monotone increasing and $O(t^2)$, then $\frac{1}{\gamma(|t|)}$ is integrable with respect to $d \mu(t)$
        if and only if $\frac{1}{\gamma(|t|)}$ is integrable with respect to $d \tilde{\mu}(t).$ That is, $f$ is $\gamma$-regular
	if and only if $f_{\text{red}}$ is.
    \end{itemize}
\end{lemma}
\begin{proof}

Let $f: \Pi \to \cc\Pi$ be analytic. Then $f$ has Nevanlinna representation $a + bz + \int_\R \frac{1}{t -z} - \frac{t}{1+t^2} \, d \mu(t)$
such that $\frac{1}{1+t^2}$ is integrable with respect to $\mu$
as in Theorem \ref{nevrep}.

Define the measure $\tilde{\mu}$ on $[-1,1]$ by \beq\label{goodmeasure} \,d\tilde{\mu(t)} = \,d\mu(t).\eeq 

Make the following definitions:
\[
f_{trivial}(z): = a + bz + \int_{\R / [-1,1]} \frac{1}{t - z}-\frac{t}{1+t^2} \, d \mu(t)
+  \int_{[-1,1]} -\frac{t}{1+t^2} \, d \mu(t).
\]
\[
f_{red}(z) := \int_{[-1,1]} \frac{1}{t - z} \, d\mu(t).
\]
That $f = f_{trivial} + f_{red}$ is obvious.

Thus (1) and (3) hold. Part (2) follows immediately from the fact that $a + bz$ is real-valued on $\R$ and that the integrals in $f_{trivial}$ represent a Pick function against the measure $\mu\big|_{\R / [-1,1]}$.

Part (A, $\Rightarrow$) follows from Lemma \ref{ajadd} and part (1). Part (A, $\Leftarrow$) follows from Lemma \ref{vacuous} and the hypothesis that $AJ_{f_{red}, \la}^0$ is bounded.

Part (B) follows from the fact that $\frac{1}{\gamma(|t|)}$ must be integrable on $\R \setminus [-1,1]$ because it is dominated by 
$\frac{1}{t^2}$ since it was assumed that $\gamma(t)$ is $O(t^2).$
\end{proof}

We now provide a simple formula for the Julia quotient, essentially as convolution of the measure with a Poisson kernel.
\begin{lemma} \label{JQint}
	Let $f: \Pi \rightarrow \overline{\Pi}$ be an analytic function of the form
	$f = \int_{[-1,1]} \frac{1}{t-z} d\mu(t)$ for some finite positive measure $\mu.$
	Then,
	\beq \label{JQintformula} J_{f} (x+iy) = \int \frac{1}{(t - x)^2 + y^2} d\mu(t).\eeq

	Note that the algebraic form of \eqref{JQintformula} implies that, $J_f(x+iy)$ is monotone for each fixed $x$ as $y$ goes to $0.$
\end{lemma}
\begin{proof}
The following computation proves the claim:
    \begin{align*}
			J_f(z)&=\frac{\IM f}{\IM z} \\
			&= \frac{1}{\IM z} \IM \int \frac{1}{t - z} \, d\mu(t) \\
			&= \frac{1}{\IM z} \int \frac{-\IM (t-z)}{|t-z|^2} \,d\mu(t) \\
			&= \int \frac{1}{ (t - x)^2 + y^2} \, d\mu(t).
	 \end{align*}
\end{proof}

\section{The augur lemma and the amortized Julia-Carath\'eodory theorem}
We now prove the central estimate relating the amortized Julia quotient to the density of the measure, which we call the \emph{augur lemma.}
Although $\gamma$-regularity is classically motivated, it is imminently plausible that the ``true theorem" is in fact the augur lemma. 
\begin{lemma}[The augur lemma]\label{AJmuBound}
    Let $\la(t)$ be $o(t).$
	Let $f: \Pi \rightarrow \overline{\Pi}$ be an analytic function of the form
	$f = \int_{[-1,1]} \frac{1}{t-z} d\mu(t)$ for some finite positive measure $\mu.$ Then there exist constants $L_1, L_2$ such that
	\begin{equation}\label{ordinarybound}
		L_1 \cdot \frac{\mu(-\ep,\ep)}{\ep \lambda(\ep)}
			\leq AJ_{f,\lambda}^0(\ep)
			\leq \frac{\mu(-2\ep,2\ep)}{\ep \lambda(\ep)}\cdot L_2 +
			4\int_{[-1,1]} \frac{1}{t^2} d\mu(t).
	\end{equation}
	We remark that the function $\frac{1}{t^2}$ will be integrable with respect to $\mu$ whenever there is any sequence going to $0$ with the Julia quotient bounded. 
\end{lemma}
\begin{proof}

	By Lemma \ref{JQint}, 
		\begin{equation}
		AJ_{f,\lambda}^0(\ep) = \frac{1}{2\ep} \int_{-\ep}^{\ep} \int_{[-1,1]} \frac{1}{(x-t)^2 + y^2} \, d\mu(t) dx,
	\end{equation}
   	    where $y= \lambda(\ep).$ We can simplify the formula by changing the order of integration and evaluating an integral via freshman calculus:
	\begin{align*}
		 AJ_{f,\lambda}^0(\ep) 
		 	&\geq \frac{1}{2\ep} \int_{-\ep}^{\ep} \int_{-\ep}^{\ep} \frac{1}{(x-t)^2 + y^2} d\mu(t) d x\\
		 	&= \frac{1}{2\ep} \int_{-\ep}^{\ep} \int_{-\ep}^{\ep} \frac{1}{(x-t)^2 + y^2} d  x d\mu(t) \\
		 	&= \frac{1}{2\ep y} \int_{-\ep}^{\ep} \arctan\left(\frac{\ep-t}{y}\right) - \arctan\left(\frac{-\ep-t}{y}\right) d\mu(t).
	\end{align*}
	Note that 
		$$\left|\frac{\ep-t}{y} - \frac{-\ep-t}{y}\right| = \frac{2\ep}{y},$$
	so for any $t \in (-\ep,\ep)$, one of $\frac{\ep-t}{y},\,\frac{-\ep-t}{y}$ is positive and one is negative, and either $\left|\frac{\ep-t}{y} \right|\geq \frac{\ep}{y}$ or $\left|\frac{-\ep-t}{y} \right|\geq \frac{\ep}{y}.$
	 Since $y=\lambda(\ep)$ is $o(\ep),$ the quantity $\frac{\ep}{y}>1$ for $\ep$ small enough. So,
	\begin{equation*}
		\arctan\left(\frac{\ep-t}{y}\right) - \arctan\left(\frac{-\ep-t}{y}\right)
			> \arctan\left(1\right)-\arctan(0).
	\end{equation*}
	Therefore,
	\begin{equation}
		L_1 \cdot \frac{\mu(-\ep,\ep)}{\ep \lambda(\ep)} \leq AJ_{f,\lambda}^0(y).
	\end{equation}

	For the upper bound consider
	\begin{align*}
		AJ_{f,\lambda}^0(\ep) 
		 	&= \frac{1}{2\ep} \int_{-\ep}^{\ep} \int_{-2\ep}^{2\ep} \frac{1}{(x-t)^2 + y^2} \, d\mu(t) dx\\
		 		&\qquad+ \frac{1}{2\ep} \int_{-\ep}^{\ep} \int_{(-2\ep, 2\ep)^{\mathsf{c}}} \frac{1}{(x-t)^2 + y^2} \, d\mu(t) dx.
	\end{align*}
	The first term can be bounded using essentially the same method as the lower bound, since the diameter of the range of arctangent is $2\pi:$
	\begin{align*}
		&\frac{1}{2\ep} \int_{-\ep}^{\ep} \int_{-2\ep}^{2\ep} \frac{1}{(x-t)^2 + y^2} \, d\mu(t) dx\\
		 	&\qquad= \frac{1}{2\ep y} \int_{-2\ep}^{2\ep} \arctan\left(\frac{\ep-t}{y}\right) - \arctan\left(\frac{-\ep-t}{y}\right) \, d\mu(t)\\
		 	&\qquad \leq \frac{1}{\ep y} \cdot \mu(-2\ep, 2\ep) \cdot \pi.
	\end{align*}

	Examining the second term, we see that $|x-t|\geq t/2$ on $(-2\ep, 2\ep)^{\mathsf{c}}\times(-\ep,\ep)$. So,
	\begin{align}
		\frac{1}{(x-t)^2 +y^2} 
			&\leq \frac{4}{t^2}  \label{integrandBound}
	\end{align}
	and we are done.
\end{proof}

We now need a technical measure theoretic lemma.
\begin{lemma}\label{isgammaaugury}
If $\frac{1}{t^2}$ is integrable with respect to $\mu,$ then
$\frac{\mu(-2t,2t)}{t}$ is $o(t).$
\end{lemma}
\begin{proof}
	To see that  $\frac{\mu(-2t,2t)}{t}$ is $o(t)$, suppose there were a sequence $t_n \rightarrow 0$ such that 
	$\frac{\mu(-t_n,t_n)}{t_n^2}> C,$ and $t_n / t_{n+1}>2.$
	Now, 
	    $$\int^1_{0} \frac{\mu(-2t,2t)}{t^3} dt$$
	must be integrable. (We have assumed $1/t^2$ is integrable with respect to $d\mu(t),$ and so, by the layer 
cake principle, $\int^1_{0} \frac{\mu(-t,t)}{t^3} dt + 2 = \int^1_{-1} \frac{1}{t^2}\,d\mu(t),$ which is comparable to the desired integral.)
	However, taking a partial Riemann sum type estimate
	with intervals $[t_n, 2t_n],$ where the integrand must be bounded below by
	$\frac{\mu(-t_n,t_n)}{8t_n^3},$
	 we see that
	    $$\int^1_{0} \frac{\mu(-2t,2t)}{t^3} dt
	    \geq \sum \frac{t_n\mu(-t_n,t_n)}{8t_n^3}
	    \geq \sum \frac{\mu(-t_n,t_n)}{8t_n^2} \geq \sum C = \infty.$$
\end{proof}

We now prove the amortized Julia-Carath\'eodory theorem, Part (1) of Theorem \ref{mainthm}.
\begin{theorem}\label{mainthmPrime}
Let $\gamma$ be  $O(t^2).$
An analytic function
$f: \Pi \rightarrow \cc \Pi$
is $\gamma$-regular
if and only if there exists 
a $\gamma$-augury $\lambda$ such that
$AJ_{f,\la}^\tau(d)$ is bounded as $d \rightarrow 0.$
\end{theorem}
\begin{proof}
    Let $\gamma(t)$ be  $O(t^2).$

    First, by assumption $\gamma$ is monotonically increasing and $\gamma(0)=0,$ which gives that $\frac{1}{\gamma}$ is monotonically decreasing and is unbounded at zero.
	Applying the layer cake principle, we see that 
	$$\int^1_{-1}\frac{1}{\gamma(C|t|)}d\mu(t) = - \int^1_{0} \mu(-t,t) \dd \frac{1}{\gamma(Ct)} + \frac{2}{\gamma(C)}.$$
	Here $\dd g$ denotes the distributional derivative of a monotone function.
    Note,
        $$-\dd \frac{1}{\gamma} = \frac{1}{\gamma^2} \dd \gamma.$$

	So, we must understand the integrability of
	$\int^1_{0} \frac{ \mu(-t,t)}{\gamma(Ct)^2} \dd \gamma(Ct).$
    	Lemma \ref{AJmuBound} implies that
	\beq \label{auguryestimate} \int^1_{0} \frac{ \mu(-t,t)}{\gamma(Ct)^2} \dd \gamma(Ct) \leq 
	\int^1_{0} \frac{t\la(t)}{\gamma(Ct)^2} \dd \gamma(Ct).\eeq
	A substitution gives that we must understand when the quantity
	$\frac{t\la(C't)}{\gamma(t)^2} \dd \gamma(t)$ integrable for some $C'>0.$ (In fact, $C' = 1/C.$)

	Therefore, if $\la$ is a $\gamma$-augury, then 
	there is a $C>0$ such that our analyzed quantity is integrable.
	Thus, $f$ is $\gamma$-regular.
	
	On the other hand, if the quantity is integrable,
	assigning $\la(t) = \frac{\mu(-2t,2t)}{t}$ gives a $\gamma$-augury.
	(Note that $\frac{t\la(t/2)\dd\gamma(t)}{\gamma(t)^2}$ is integrable using the same layer cake calculation,
	since we obtain equality up to a constant multiple in Equation \eqref{auguryestimate}.)
	By Lemma \ref{isgammaaugury}, $\la(t)$ is $o(t).$
	    
	 Now,
	       $\frac{\mu(-2\ep, 2\ep)}{\ep \la(\ep)} = 1.$
	 So, by Lemma \ref{AJmuBound}
	       $AJ^0_{f, \la}(\ep) \leq L\frac{\mu(-2\ep,2\ep)}{\ep\la(\ep)} +C = L+C,$ which is bounded.
\end{proof}

One should note that the construction of $\la$ in the converse direction gives that for a given specific function, to test any level of $\gamma$-regularity, the \emph{same} $\la$-Stolz region can be chosen. Note that the lower bound of the augur lemma, Lemma \ref{AJmuBound}, means that on any essentially larger $\la$-Stolz region (up to constants) the amortized Julia quotient will be unbounded.

\section{Interpretations of $\gamma$-regularity and pitting irregularities} \label{pitting}

    We show that $\gamma$-regularity gives regularity on certain (probably small) $\lambda$-Stolz regions. (Where regularity is interpreted variously as having boundedness of the Julia quotient, boundedness of the amortized Julia quotient, and boundedness of the function.)
 However, outside of these good $\lambda$-Stolz regions, there are families of examples arising from functions where the measure $\mu$ in the representation in Equation \eqref{fullnev} is a sum of point masses. Directly above the point masses in $\Pi$, the behavior of $J_f$ is very bad, which helps to explain why we need amortization in cases where $\gamma(t)/t$ is $o(\sqrt{\gamma(Ct)})$ for all $C>0.$ Essentially, above a point mass near the boundary, the value very briefly fluctuates wildly, creating a narrow but deep ``pit" which is filled in by the amortization process. 

\subsection{Boundedness of the Julia quotient}
The following omnibus theorem gives various places where the Julia quotient and its amortizations are bounded.
    \begin{theorem} \label{pittingtheorem}
	Let $\gamma$ be  $O(t^2)$ be a monotone increasing function.
        \begin{enumerate}
        \item  For any $\gamma$-regular function $f: \Pi \rightarrow \cc \Pi,$
        there exists a $C>0$ such that
        $J_{f}(z)$ is bounded on $S^{\sqrt{D\gamma(Ct)}}_{\tau}$ as $z\rightarrow 0$
        for every $D>0.$
        \item  For any $\gamma$-regular function $f: \Pi \rightarrow \cc \Pi,$
        there exists a $C>0$ such that
        $AJ^\tau_{f,\gamma(Ct)/Ct}(d)$ is bounded as $d\rightarrow 0.$
        (Note that this implies that if
        we can find a $\gamma$-augury $\lambda$ such that 
        the amortized Julia quotient is bounded near $\tau$ on the $\lambda$-Stolz region,
        then it is also bounded on the $\min(\gamma(Ct)/Ct,\lambda)$-Stolz region.)
         \item Moreover, for any $\lambda$ which is  $o(\sqrt{C\gamma})$ for some $C>0$,
        there is a function $f: \Pi \rightarrow \cc \Pi$ that  
        is $\gamma$-regular, and 
        $J^0_{f}(z)$ is unbounded on $S^{\lambda}_{\tau}$ as $z\rightarrow 0.$
        \end{enumerate}
    \end{theorem}
    \begin{proof}
            (1): Without loss of generality, assume
                $\frac{1}{\gamma(t/2)}$ is integrable with respect to $\mu(t).$
            Let $z\in S^{\sqrt{D\gamma(t)}}_{\tau}.$
            Without loss of generality, by Lemma \ref{splitupterms},
            $$J_{f}(z) = \int_{[-1,1]} \frac{1}{(x-t)^2 + y^2} \, d\mu$$
            So, since $y \geq \sqrt{\gamma(x)}$, and using the bound of \eqref{integrandBound}, we have
                $$J_{f}(z) \leq
                \int_{[-1,1]} \frac{1}{(t/2)^2}\, d\mu(t) + 
                \int_{[-2x,2x]} \frac{\gamma(t/2)}{\gamma(x)} \frac{d \mu(t)}{\gamma(t/2)}.$$
            Note,
                $$\frac{d\mu(t)}{\gamma(t/2)}$$
            is a finite measure. So we see that 
            $J_{f}$ is bounded on $S^{\sqrt{D\gamma(t)}}_{\tau}.$
            
            \smallskip
            
            (2):
            Without loss of generality,
            $\int_{[-1,1]} \frac{1}{\gamma(|t|)}\dd \mu(t)$ exists  and is finite.
            Now consider
            $$AJ^\tau_{\lambda, f}(t) \leq D\frac{\mu(2t,-2t)}{t\gamma(2t)/2t} + C$$
            by Lemma \ref{AJmuBound}.
            Since $1/\gamma(|t|)$ is integrable with respect to $\mu,$ it must be that
            $\mu(-t,t) / \gamma(t)$ is bounded, since it gives a lower bound for the integral $\int_{[-1,1]} \frac{1}{\gamma(|t|)}\dd \mu(t).$
            
            \smallskip
            
            (3): Consider a function $\lambda$ which for some $C>0$ is $o(\sqrt{C\gamma}).$
	Without loss of generality $C=1.$
            We want to construct an $f$ which is $\gamma$-regular and that has
            $\frac{1}{\gamma(t)}$ is integrable with respect to the corresponding $\mu(t).$
            Take a sequence $t_n\rightarrow 0$ such that for every $C>0,$
            $\gamma(t_n)/\lambda(t_n)^2 \geq 2^n n^2.$
            Define a measure $\mu = \sum \frac{\gamma(t_n)}{n^2}\delta_{t_n}.$
		Note that $$\int \frac{1}{\gamma(t)} \dd \mu(t) = \frac{\pi^2}{6}<\infty.$$
	    So, the corresponding $f$ is, in fact, $\gamma$-regular.
		Further note:
            $$J_{f}(z) \geq \int_{[-1,1]} \frac{1}{(x-t)^2 + y^2} d \mu(t).$$
            So,
            $$J_{f}(t_n+ i\lambda(t_n)) \geq \frac{\gamma(t_n)}{\lambda(t_n)^2n^2} \geq 2^n,$$
            which tends to infinity, and we are done.
    \end{proof}
    
    We immediately obtain Part (2) of Theorem \ref{mainthm}, a ``perfect"
	Julia-Carath\'eodory type theorem.
    \begin{corollary}\label{supercor}
        Let $\gamma$ be  $O(t^2)$ such that $\sqrt{\gamma(t)}$ is a 
        $\gamma$-augury.
	Let $f: \Pi \rightarrow \cc \Pi$ be an analytic function. The following are equivalent:
	\begin{enumerate}
		\item the function $f$ is $\gamma$-regular,
		\item $J_{f}^\tau(z)$ is bounded on 
                $S^{\sqrt{D\gamma(Ct)}}_{\tau}$ as $z\rightarrow 0$ for some $C,D>0,$
		\item
		$$\limsup_{t\rightarrow 0} J_{f}^\tau(t+ iD\gamma(C|t|)) < \infty.$$
	\end{enumerate}
    \end{corollary}
	The equivalence of (1) and (2) follows directly from part (1) of Theorem \ref{pittingtheorem}, and the equivalence of (2) and (3) follows from
	the fact that $J_f(z)$ is monotone on vertical lines as we approach the real axis for functions as in Lemma \ref{JQint}.

	Moreover, we also obtain Part (3) of Theorem \ref{mainthm} as a corollary of Theorem \ref{pittingtheorem}, which grants us the ability to preordain the choice of $\gamma$-augury.
    \begin{corollary}\label{uniformcor}
        Let $\gamma$ be  $O(t^2)$ such that $\gamma(t)/t$ is a 
        $\gamma$-augury.
            An analytic function
                $f: \Pi \rightarrow \cc \Pi$
                is $\gamma$-regular
                if and only if
                $AJ_{f,\gamma(Ct)/Ct}^\tau(d)$ is bounded on 
                $S^{\gamma(Ct)/Ct}_{\tau}$ as $d\rightarrow 0$ for some $C>0.$
    \end{corollary}
  
\subsection{$\gamma$-horocyclic continuity}
Along the lines of Theorem \ref{pittingtheorem}, one can establish the $\gamma$-analogue of \emph{horocyclic continuity} of the function.
We say a function is \dfn{$\gamma$-horocyclically continuous} whenever
for each $D>0,$
		\[\sup_{S^{D\gamma(Ct)}_{\tau}\cap B(\tau, 1/D)} |f(z)-f(\tau)|<\infty,\]
and, moreover,
 $$\lim_{D\rightarrow \infty} \sup_{S^{D\gamma(Ct)}_{\tau}\cap B(\tau, 1/D)} |f(z)-f(\tau)| = 0$$
for some $C>0.$
We will show that any $\gamma$-regular function is $\gamma$-horocyclically continuous, and moreover that
for any $\la(t)$ which is $o(\gamma(Ct))$ for every $C>0$ { there can be no such guarantee.}
Classical horocyclic continuity, where $\gamma(t)=t^2$ was established using the Julia inequality \ref{thejuliainequality} in \cite{ju20}.

\begin{theorem}\label{horotheorem}
Let $\gamma$ be  $O(t^2)$ be a monotone increasing function.
\begin{enumerate}
\item  For any $\gamma$-regular function $f: \Pi \rightarrow \cc \Pi,$ $f$ is $\gamma$-horcyclically continuous. That is, 
        there exists a $C>0$ such that
	for each $D>0,$
		\[\sup_{S^{D\gamma(Ct)}_{\tau}\cap B(\tau, 1/D)} |f(z)-f(\tau)|<\infty,\]
	and, moreover,
        $$\lim_{D\rightarrow \infty} \sup_{S^{D\gamma(Ct)}_{\tau}\cap B(\tau, 1/D)} |f(z)-f(\tau)| = 0.$$
\item For any $\lambda$ which is $o(\gamma(Ct))$ for some $C>0$, there is a $\gamma$-regular function $f: \Pi \rightarrow \cc \Pi,$
such that
        $$\sup_{S^\lambda_{\tau}} |f(z)-f(\tau)| = \infty.$$
\end{enumerate}
\end{theorem}
\begin{proof}
 Without loss of generality, let $\tau =0$.

(1):
We will show that for an analytic function $f: \Pi \rightarrow \cc \Pi$ such that 
$\int \frac{1}{\gamma(t)} \dd\mu(t) = 1$ that the claim holds, where $\mu$ is the measure corresponding to $f$.
The condition is convex, so it is sufficient to check that the claim holds at the extreme points $\mu = \gamma(t_0)\delta_{t_0}$
which give rise to the corresponding functions $\frac{\gamma(t_0)}{t_0-z}.$
It is obvious that
		$$\lim_{D\rightarrow \infty} \sup_{S^{D\gamma(Ct)}_{0}\cap B(0, 1/D)} \abs{\frac{\gamma(t_0)}{t_0-z}-\frac{\gamma(t_0)}{t_0}} = 0$$
for each such $t_0$. However, one must show that \[M_{D}(t_0) = \sup_{S^{D\gamma(Ct)}_{0}\cap B(0, 1/D)} \abs{\frac{\gamma(t_0)}{t_0-z}-\frac{\gamma(t_0)}{t_0}}\]
is uniformly bounded in $t_0$ for each $D$ for that statement to have content. (Namely,
in general,
$ \sup_{S^{D\gamma(Ct)}_0\cap B(0, 1/D)} |\int \frac{1}{t-z} - \frac{1}{t}\dd \mu(t)|\leq \int M_{D}(t)\dd \mu(t),$ and so, if the
right hand side is integrable, we win by an application of the monotone convergence theorem.)

Choose $C = 2.$
Consider 
$$\sup_{S^{D\gamma(2t)}_{0}\cap B(0, 1/D)} \abs{\frac{\gamma(t_0)}{t_0-z} - \frac{\gamma(t_0)}{t_0}}.$$
We want to maximize $\abs{\frac{\gamma(t_0)z}{(t_0-z)t_0}}.$
Write $z = x+iy$ so that
$$\abs{\frac{\gamma(t_0)z}{(t_0-z)t_0}}^2 = \frac{\gamma(t_0)^2}{t_0^2} \frac{x^2 + y^2}{(t_0-x)^2+y^2}.$$
The argumentation will go by cases.
If $y>|x|,$ then the quantity is obviously bounded.
If $y \leq |x|$ and $|x| \leq \frac{3}{2}t_0,$ then, since $y>D\gamma(t_0),$ as $z \in S^{D\gamma(2t)}_{0},$ we see that
$$\abs{\frac{\gamma(t_0)z}{(t_0-z)t_0}}^2\leq \frac{\gamma(t_0)^2}{t_0^2} \frac{\frac{9}{2}t_0^2}{D^2\gamma(t_0)^2}
\leq \frac{9}{2D^2}.$$
If $y \leq |x|$ and $|x| > \frac{3}{2}t_0,$ then, we see that
$$\abs{\frac{\gamma(t_0)z}{(t_0-z)t_0}}^2\leq \frac{\gamma(t_0)^2}{t_0^2} \frac{2x^2}{(x/2)^2} \leq  2\frac{\gamma(t_0)^2}{t_0^2}$$
which is bounded by the assumption that $\gamma(t)$ is $O(t^2).$
 
\smallskip

(2): { The construction is similar to the construction proving Theorem \ref{pittingtheorem} Part (3).}
Let $\la$ be $o(\gamma(Ct))$ for some $C> 0$. Fix a sequence $t_n\rightarrow 0$ such that $\gamma(t_n)/\lambda(t_n) > 2^n n^2$.
Now construct the measure $\mu = \sum \frac{\gamma(t_n)}{n^2}\delta_{t_n}$, and let 
$f(z) = \int \frac{1}{t-z} \dd \mu(t).$
Then \[\abs{f(t_n+ i\lambda(t_n))} \geq \frac{\gamma(t_n)}{\lambda(t_n)n^2} > 2^n\]
since  $$\frac{\gamma(t_n)}{\lambda(t_n)n^2} = \text{Im } \left(\frac{1}{t_n-z}  \frac{\gamma(t_n)}{n^2}\right)\bigg|_{z = t_n+ i\lambda(t_n)}.$$
\end{proof}

Equipped with the horocyclic continuity theorem one can prove the following, which says that $\gamma$-regularity is preserved under reasonable compositions, including automorphisms of the upper half plane which do not move $f(\tau)$ to infinity.
\begin{theorem}\label{perturbationtheorem}
Let $\gamma$ be  $O(t^2)$ be a monotone increasing function such that $\gamma$ is a $\gamma$-augury.
Let $f: \Pi \rightarrow \cc \Pi$ which has a nontangential limit $f(\tau).$
Let $U$ be a neighborhood of $f(\tau).$
Let $g: \Pi \cup U \rightarrow \mathbb{C}$ be a nonconstant analytic function such that
$g(\Pi) \subseteq \cc \Pi$ and $g|_{\mathbb{R}\cap U}$ is real-valued.
Then, $f$ is $\gamma$-regular if and only if $g\circ f$ is $\gamma$-regular.
\end{theorem}
\begin{proof}
	Without loss of generality $\tau=f(\tau)=g\circ f(\tau ) =0.$

	Note that if $f$ is $\gamma$-regular then there exists a $\gamma$-augury $\lambda$ such that 
	$AJ^\tau_{f,\lambda}(d)$ is bounded as $d \rightarrow 0$ by Theorem \ref{mainthm}, and that since $\gamma$ is a $\gamma$-augury
	we can assume $\lambda \geq \gamma.$
	
	There is a neighborhood $V\subset U$ of $f(\tau)$ such that $g(z) \approx bz.$
	Therefore, whenever $f(z)\in V,$ $bJ_{f}(z) \approx J_{g\circ f}(z)$ by an algebraic calculation.
	
	By Theorem \ref{horotheorem}, there exists $D$ and $C$ such that $f(S^{D\gamma(Ct)}_{\tau}\cap B(\tau, 1/D))\subset V.$
	Without loss of generality $D=C =1.$
	Now $S^{\lambda}_0 \subseteq S^{\gamma}_{0}$ so we are done, since for small
	$d,$ $bAJ^\tau_{f,\lambda}(z) \approx AJ^\tau_{g\circ f,\lambda}(z).$
\end{proof}

\begin{figure}[H]
	\caption{An illustration of Theorem \ref{pittingtheorem} and Theorem \ref{horotheorem}. A $t^3$-regular function is bounded on the orange region bounded by $y=|x|^3,$ which is overlayed with a periwinkle region where the
	amortized Julia quotient is bounded, with boundary curve $y=|x|^{2},$ which is itself overlayed by a blue region where the Julia quotient is properly bounded, with boundary curve $y=|x|^{3/2}.$}
\begin{tikzpicture}[scale=.7]
    \begin{axis}[
        domain=-1.25:1.25,
        xmin=-.5, xmax=.5,
        ymin=-.2, ymax=.5,
        samples=401,
        axis y line=none,
        axis x line=middle,
	ticks =none,
        legend style={at={(1,.7)},anchor=north west},
    	]
        \addplot+[mark=none, black, fill=Apricot] {abs(x)^3};
        \addplot+[mark=none, black, fill=Periwinkle] {x^2};
	\addplot+[mark=none, black, fill=Aquamarine] {abs(x)^(3/2)};
    \end{axis}
\end{tikzpicture}
\end{figure}
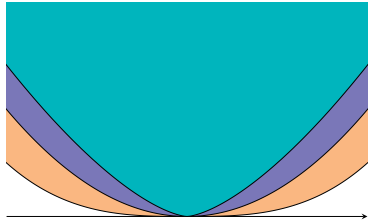

\begin{figure}[H]
	\caption{In the classical case it is known that the Julia quotient is bounded on Stolz regions, (in blue) and that the function is bounded on regions of quadratic approach (in orange). This corresponds to $t^2$-regularity. In this case, $t= \gamma(t)/t = \sqrt{\gamma(t)}$ so the automatic boundedness of the Julia quotient derived from Theorem \ref{pittingtheorem} is the same as that for the amortized Julia quotient.}
\begin{tikzpicture}[scale=.7]
    \begin{axis}[
        domain=-1.25:1.25,
        xmin=-.5, xmax=.5,
        ymin=-.2, ymax=.5,
        samples=401,
        axis y line=none,
        axis x line=middle,
	ticks=none,
        legend style={at={(1,.7)},anchor=north west},
    	]
        \addplot+[mark=none, black, fill=Apricot] {abs(x)^2};
				\addplot+[mark=none, thick, black, fill=Aquamarine] {abs(x)};
    \end{axis}
\end{tikzpicture}
\end{figure}
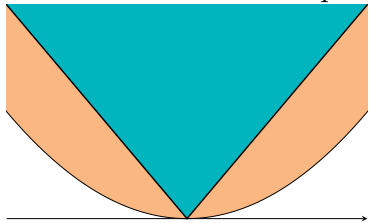

\section{The extended Bolotnikov-Kheifets theory using the fractional Laplacian} \label{iterated}
\subsection{Background}

In the case that $D = \Omega = \D$, the Julia quotient is infinitely subharmonic. That is, for analytic functions $\ph:\D \to \cc\D$, we have $\Delta^n J_\ph(z) \geq 0.$ The subharmonicity comes from the fact that
$A(z,w)=\frac{1-\ph(z)\cc{\ph(w)}}{1-z\cc{w}}$ is a positive kernel, and thus, on compact subsets of its domain, $A$ can be written as $A(z,w) = \sum f_k(z)\cc{f_k(w)}$ by Mercer's theorem. In Theorem 1.2 of \cite{bk06}, Bolotnikov and Kheifets established the exact relationship between nontangential boundedness of the positive quantity $\Delta^n J_\ph(z)$ and regularity to order $2n+1.$ We present their main result here in our language. (We note that regularity to order $2n+1$ was originally phrased as positivity condition for a certain matrix constructed from the Taylor coefficients of $\tilde\ph$ at $\tau.$)

\begin{theorem}[Bolotnokov, Kheifets \cite{bk06}]\label{bolo}
Let $\ph: \D \to \cc\D$ be an analytic function. Let $\tau$ be a point in $\T = \partial \D$. The following are equivalent:
\begin{enumerate}
    \item There exists a sequence $\la_n \subset \D$ tending to $\tau$ such that $\Delta^n J_\ph(\la_n)$ is bounded as $\la_n \to \tau$;
    \item for every sequence $\la_n$ in $\D$ tending to $\tau$ nontangentially, the sequence $\Delta^n  J_\ph(\la_n)$ is bounded as $\la_n \nt \tau$;
    \item the function $\ph$ is regular to order $2n+1$ at $\tau.$
\end{enumerate}
\end{theorem}

 Unlike the Laplacian approach taken by Bolotnikov and Kheifets on the disk, our main result, Theorem \ref{mainthm}, characterizes boundary regularity entirely in terms of the classical Julia quotient, with the trade-off that we must work on larger than nontangential sets.

\subsection{The extended Bolotnikov-Kheifets theory}
	Consider the space $C_{r}$ of functions of the form
		$$\int \frac{1}{|z-t|^{2r}} d\mu,$$
	where $\mu$ is a finite signed Borel measure.
	The \dfn{fractional Laplacian} $\Delta^s: C_r \rightarrow C_{r+s}$
		is defined to be
		$$\Delta^s \int \frac{1}{|z-t|^{2r}} d\mu = \frac{\Gamma(r+s)^2}{\Gamma(r)^2}\int \frac{1}{|z-t|^{2(r+s)}} d\mu.$$
	We define $\Delta^s C = 0$ for any constant $C.$
	(Here, we are using the fact that $\Delta = \frac{\partial^2 }{\partial z \partial \cc z}$
	and that  $|z-t|^{2r} = (z-t)^r\cc{(z-t)^r}$, for which the behavior
	on $s\in \mathbb{N}$ justifies the definition on non-integral $s.$)
	
	The Julia quotient of an analytic function 
	 $f: \Pi \rightarrow \overline{\Pi}$ can be shown to be in the domain of
	  $\Delta^s.$ The following gives the fractional Laplacian of the Julia quotient for an analytic function of the form
	$f = \int_{[-1,1]} \frac{1}{t-z} d\mu(t)$ for some finite positive measure $\mu.$
\begin{lemma}\label{iterJQ}
Let $f: \Pi \rightarrow \overline{\Pi}$ be an analytic function of the form
	$f = \int_{[-1,1]} \frac{1}{t-z} d\mu(t)$ for some finite positive measure $\mu.$ Then
	\begin{equation} 
		\lapl^s J_{f}(x+iy) = \Gamma(s)^2 \int_{-1}^{1} \frac{1}{\left((x-t)^2 +y^2\right)^{s+1}}\,d\mu(t).
	\end{equation}
\end{lemma}

We obtain the following estimate, similar to what we obtained before in the ordinary case in Theorem \ref{AJmuBound}.
\begin{lemma}
Fix $s\geq 0.$
Let $\la$ be $o(t).$
Let $f$ of the form $\int_{[-1,1]} \frac{1}{t-z}\,d\mu(t)$.
Then, 
	\begin{equation}
		\lapl^s AJ_{f,\lambda}^0(\eps) = \frac{\Gamma(s)^2}{2\eps} \int_{\eps}^{\eps} \int_{-1}^{1} \frac{1}{\left((x-t)^2 +y^2\right)^{n+1}}\,d\mu(t)\,dx.
	\end{equation}
	Moreover, there exist constants $L_1, L_2, C$ such that
	\begin{equation}\label{BKversion}
		L_1 \,\cdot \,\frac{\mu(-\eps, \eps)}{\eps \lambda^{2s+1}(\eps)}
			\leq \lapl^s AJ_{f,\lambda}^0\leq
			L_2 \,\cdot \frac{\mu(-\eps, \eps)}{\eps \lambda^{2s+1}(\eps)} +C\int_{[-1,1]} \frac{1}{t^{2s+2}} d\mu(t).
	\end{equation}
\end{lemma}

The proof essentially follows the same arguments as the proofs of Lemma \ref{JQint} and Lemma \ref{AJmuBound} by
changing the order of integration and doing the obvious trigonometric substitution.

So we arrive at the following theorem, along the lines of Theorem \ref{mainthm},
in analogy with Bolotnikov and Kheifets \cite{bk06}. (Note the algebraic similarity of \eqref{BKversion} to \eqref{ordinarybound}.)
\begin{theorem}
Let $\gamma$ be  $O(t^{2s+2}).$
An analytic function
$f: \Pi \rightarrow \cc \Pi$
is $\gamma$-regular
if and only if there exists 
a $\gamma$-augury $\lambda^{2s+1}$ such that $\la(t)$ is $o(t)$ and 
$\Delta^s AJ_{f,\la}^\tau(d)$ is bounded as $d \rightarrow 0.$
\end{theorem}

\section{Commentary}

We point out some related open problems and connections to other work. 
\subsection{Schur class}
In \cite{bk06}, Bolotnikov and Kheifets get results on the disk about members of the Schur class. Thus, there should be obvious interest in transforming Theorem \ref{mainthm} back to an analogous theorem on the disk. Attempting to go directly through the conformal map from $\Pi$ to $\mathbb{D}$ causes some amount of confusion due to the polar singularity of such a map not playing well with amortization, though we have partial results in this direction. There is also a body of followup and related results in Bolotnikov-Kheifets theory that may have analogues in our setting (e.g. \cite{MR2154356, MR2448987, MR2522782, MR3495433, MR2390675}).
\subsection{General kernels}
What is important in our analysis of the Julia quotient is essentially that self-maps of the upper half plane
are of the form $\int \frac{1}{t-z}\,d\mu(t).$ That is, such functions can be viewed as the convolution of 
$\frac{1}{z}$ with a measure $\mu.$ Moreover, the Julia quotient, in this case, is given as the convolution (in $x$)
of $\frac{1}{x^2+y^2}$ with $\mu.$ Similarly, the extended Bolotnikov-Kheifets theory relies on the fractional Laplacian of the Julia quotient being given by the convolution
of $\frac{1}{(x^2+y^2)^n}$ with $\mu.$
In general, one imagines that functions arising from convolutions of the form $\int k(t-z)\mu(t)$ for nice functions $k$ should have an
analogous theory of $\gamma$-regularity.

\subsection{Analysis of Cauchy transforms}
The notion of $\gamma$-regularity allows for finer resolution in analysis of the Cauchy transform of a measure $\mu$, $\int \frac{1}{t-z}\dd \mu(t).$ In turn, that could lead to insight into finer aspects of unbounded self-adjoint operators via the spectral theorem. (Evaluation of various Julia type quotients and nontangential limits are important in perturbation theory for self-adjoint operators, for example. See e.g. \cite{liawtreil2018} and references contained therein for a good sampling. A survey of the theory of the Cauchy transform and recent developments in perturbation theory can be found in \cite{rosscimamath}.)
\subsection{Several variables}
A large body of work generalizes the classical Julia-Carath\'eodory theory to analytic functions in two or more variables in both the commutative and free noncommutative settings (e.g. \cite{pascoePEMS, mprevisit, amy10a, amy10b, aty12, ptdpick, tdopmat}) with methods reliant on non-tangential approach. Generalizing Theorem \ref{mainthm} to several variables would open these questions up to finer investigation of boundary regularity. Resolvent methods as in \cite{aty12} may present difficulty in integration and thus establishing some fundamental lemmas, such as the augur lemma. Integral methods, such as those in \cite{lugarnedic}, may prove more tractable, although the understanding of the relationship between the constructed measures and the geometry of the function would need to be developed.

\subsubsection{Free probability}
Of specific interest in several variables are applications to noncommutative function theory and free probability. The connections with this work to the extensive body of seemingly related results in moment theory, while considered in a preliminary fashion here in Section \ref{momentSection}, should be explored more deeply. In the free case, a combination of the results about the representation of free Pick functions as certain Cauchy transforms in \cite{pastdcauchy} and the Hankel vector moment sequence approach in \cite{pascoePEMS} may provide an avenue to a foundation for moment determinacy theory in operator-valued free probability. 
%\subsection{Geometric moment determinacy}

\section{Acknowledgments}

The authors would like to thank Constanze Liaw for pointing out some helpful connections.

\bibliography{references}
\bibliographystyle{plain}

\printindex

\end{document}